\documentclass[12pt,a4paper]{article}

\usepackage[T2A]{fontenc}
\usepackage{mathrsfs}

\usepackage[utf8]{inputenc}
\usepackage[english]{babel}
\usepackage{indentfirst}
\usepackage{misccorr}
\usepackage{graphicx}
\usepackage{amsmath}
\usepackage{amssymb}
\usepackage{amsthm}
\usepackage{hyperref}
\usepackage{color}

\newif\ifru
\rufalse

\newif\ifen
\entrue % \*true or \*false
\newtheorem{theorem}{Theorem}

\newtheorem{Lemma}{Lemma}

\ifen
\title{{\normalsize\tt\hfill\jobname.tex}\\
On strong law of large numbers for non identically distributed random variables
}

\author{I.V. Kozlov\footnote{M.V. Lomonosov Moscow State University; Institute for Information Transmission Problems (A.A.Kharkevich Institute); Vega Institute Foundation; email: ilia.kozlov@math.msu.ru}, \; A.Yu. Veretennikov\footnote{Institute for Information Transmission Problems (A.A.Kharkevich Institute); RUDN University; NRU HSE; MSU; email: alexander.veretennikov2011@ya.ru}
}
\fi

\ifru
\title{{\normalsize\tt\hfill\jobname.tex}\\
Об усиленном  законе больших чисел для попарно независимых случайных величин
}

\author{И.В. Козлов\footnote{Московский государственный университет им. М.В. Ломоносова; Институт проблем передачи информации им. А.А. Харкевича; email: } \; А.Ю. Веретенников\footnote{Институт проблем передачи информации им. А.А. Харкевича; РУДН ; email: alexander.veretennikov2011@ya.ru}
}
\fi

\date{}

\begin{document}
\maketitle

\ifen
\begin{abstract}
\noindent
%Some new remarks concerning strong law of large numbers (SLLN) for pairwise independent random variables are presented. The main point is relaxing the assumption on the existence of a mean value for each random variable.

A new  version of a strong law of large numbers for a ``good'' pairwise independent sequence of random variables (r.v.'s) with a small part of ``bad'' dependent r.v.'s is proposed. The main goal is to relax the assumption on the existence of the expectation for each summand: the members of an ``infrequent'' part of the whole sequence may have moments of orders converging to zero.

\medskip


\medskip


\end{abstract}
\fi

\ifru
\begin{abstract}
\noindent

Предложена новая версия усиленного закона больших чисел для ``хорошей'' последовательности случайных величин (с.в.) с небольшой частью ``плохих'' зависимых с.в. Основной целью являлось ослабление требования на существование математического ожидания для каждого слагаемого: члены ``разреженной'' части всей последовательности может иметь моменты порядков, сходящихся к нулю.

\medskip

\noindent
Ключевые слова: усиленный закон ольших чисел; попарная независимость; ччастичная зависимость; частичное отсутствие математических ожиданий.

\medskip


\end{abstract}
\fi

\ifen
\section{Introduction}
The Law of Large Numbers (LLN for what follows) is the first rigorously proved limit theorem in probability. The interest to any new its version is natural because the majority of  methods of parameter estimations and hypothesis testing is based on LLN, but also due to a philosophical importance of such theorems. The history of the LLN goes back to the work by Jakob Bernoulli in 1713, see \cite{Bernoulli},  then continued by de Moivre, Laplace, Chebyshev \cite{Chebyshev} et al. In the 20th century the strong form of an LLN was discovered by Borel \cite[chapter II, section 11]{Borel1909}, Cantelli  \cite[Lemma 2 \& formulae (28)-(29)]{Cantelli}, Mazurkiewicz  \cite{Mazurkiewicz}, Rajchman \cite{Rajchman32} (see also a recent exposition of the latter work in \cite{Chandra3}), which was brought to a nearly final form by Kolmogorov \cite{Kolmogorov}, yet further developed by Etemadi  \cite{Etemadi81}, Cs\"orgo, Tandori, and Totik \cite{Hungarians}, Landers and Rogge \cite{Rogge, Rogge2}, Chandra \cite{BoseChandra94, Chandra, Chandra2, Chandra_book, Chandra3}, among others. The history may be recalled in \cite{Seneta92, Seneta3, Shiryaev}.

The goal of the present note is to offer some further  extension for the case of pairwise independent random variables, which includes Kolmogorov's SLLN \cite[theorem 6.5.3]{Kolmogorov}, as well as the recent result in \cite{AA_AV24}. The matter is that a certain, not necessarily finite part of the summands may have infinite expectations with moments of orders, which may converge to zero. The most important technical steps relate to the calculi in \cite{Kolmogorov, MZ, Sawyer69}. 

The present paper consists of four sections: this introduction, the setting, the SLLN for partially pairwise independent random variables (theorem 
\ref{thm1}), the proof of theorem \ref{thm1}.
\fi

\ifen
\section{The setting} %Классические результаты}
Let us consider a sequence $(X_n, n\ge 1)$ of pairwise independent random variables (r.v's) with finite expectations, for which without loss of generality the equalities 
\begin{equation}\label{eq2-1}
\mathsf E X_i = 0, \qquad i\ge 1,
\end{equation}
hold. It is also assumed that there exists $n_0 \ge 1$ such that 
\begin{equation}\label{CUI}
C(n_{0})
  := \int_{0}^{\infty}
       \sup_{n \ge n_{0}}
       \frac{1}{n}
       \sum_{k = n_{0}}^{n}
      \mathsf P\!\bigl(|X_{k}| > x\bigr)\,dx
    < \infty.
\end{equation}

Moreover, there is another sequence of r.v's $(Y_n)$, arbitrarily dependent and, in general, without finite expectations. Instead, it is assumed that there is a nonrandom sequence of values $(a_n)$, such that all $a_n>0$ and $a_n\downarrow 0, \, n\to\infty$, and $\mathsf E |Y_n|^{a_n}<\infty$; an additional technical condition on the sequence $(Y_n)$ is as follows: let 
\begin{equation}\label{defVn}
    V_n := |Y_n|^{a_n}, \quad \mathsf EV_n < \infty,
\end{equation}
and there exists a nonincreasing function $\bar G(t), \, t\ge0$ such that 
\begin{equation}\label{VnCG}
  \sup_{n\ge 1}\mathsf P\bigl(V_n>t\bigr)\le\bar G(t), \qquad
  \int_{0}^{\infty}\bar G(t)dt =: C_G <\infty.
\end{equation}

Now, the sequence of r.v's $(Z_n)$ is constructed as a mixture of $(X_n)$ and $(Y_n)$. More precisely, there is one more nonrandom sequence $(\alpha_n)$ taking values $0$ and $1$; denote $\varphi_n:= \sum_{k=1}^n \alpha_k$ and $\psi_n:= n - \varphi_n$, and, finally, define 
$$
Z_n:= Y_{\varphi_n} 1(\alpha_n = 1) + X_{\psi_n} 1(\alpha_n = 0).
$$
The law of large numbers (LLN in what follows) will be stated and proved for the sequence $(Z_n)$ under the standing condition that the sequence $(Y_n)$ is ``infrequent'', see what follows. Notice that recently a strong LLN was established in a likewise case under a more restrictive condition that all $a_n\equiv a \in (0,1)$, see \cite{AA_AV24}.

The condition of ``infrequency'' of the sequence $(Y_n)$ is as follows:

\begin{equation}\label{finan}
\limsup_{n \to \infty} \frac{\varphi_n^{\frac{1}{a_{\varphi_n}}}}{n} < \infty \quad (\text{i.e., } \exists \; C>0: \; \varphi_n^{\frac{1}{a_{\varphi_n}}} \le Cn \; \text{for  $n$ large enough.}) \end{equation}
Naturally, we wish that the sequence $(\varphi(n))$ increase to infinity (otherwise it is not an interesting case):
\begin{equation}\label{anlnn}
\varphi(n) \to \infty, \quad n\to\infty. 
\end{equation}
The condition (\ref{anlnn}) will also be assumed, although, formally, the result will hold for a bounded sequence $(\varphi(n))$, too.

\begin{theorem}\label{thm1}
The assumptions (\ref{eq2-1})–(\ref{anlnn}) imply an SLLN 
\begin{equation}\label{thmSLLN}
\frac{1}{n}\sum_{k=1}^{n} Z_k \xrightarrow{\text{a.s.}} 0.
\end{equation}
\end{theorem}
\fi

\ifen
\section{Proof of Theorem}
We will need two lemmata. One is a simple analogue of N. Etemadi's theorem from  \cite{Etemadi81} under relaxed Cesàro UI condition, the other is an advanced analogue of the a calculus  from \cite{Sawyer69}, see also \cite{Korchevsky}.

\begin{Lemma}[for $(X_n)$]\label{lem:X}
Let the assumptions (\ref{eq2-1})–-(\ref{CUI}) be satisfied. Then
\begin{equation}\label{le1-eq1}
\frac{1}{n}\sum_{k=1}^{n-\varphi(n)} X_k \xrightarrow{\text{a.s.}} 0,
\qquad n\to\infty.
\end{equation}
\end{Lemma}

\begin{proof}
First of all, we have for each $n_0$, 
\[
\frac{1}{n}\sum_{k=1}^{n_0-1} X_k \xrightarrow{\text{a.s.}} 0.
\qquad n\to\infty,
\]
Hence, it suffices to verify \[
\frac{1}{n}\sum_{k=n_0}^{n-\varphi(n)} X_k \xrightarrow{\text{a.s.}} 0,
\qquad n\to\infty.
\]
%{\color{red}The latter follows from theorem 1 in \cite{Etemadi81}.}
The latter follows from the results of \cite{BoseChandra94}; see also Theorem~2 in \cite{Korchevsky}.
\end{proof}

\begin{Lemma}[for $(Y_n)$]\label{lem:Y}
Let assumptions  (\ref{defVn}) -- (\ref{finan})  be satisfied. Then
\begin{equation}\label{le2-eq1}
\frac{1}{n}\sum_{k=1}^{\,\varphi(n)}Y_k \xrightarrow{\text{a.s.}} 0,
\qquad n\to\infty.
\end{equation}
\end{Lemma}
\begin{proof} Let $\tilde V_n := \min\{V_n, n\}$. Let us show that for any $p>1$, 
\[
  \mathsf E\sum_{n=3}^{\infty}\frac{\tilde V_n^{p}}{n^{p}} < \infty.
\]
We have, 
\begin{align*}
& \mathsf E\tilde V_n^{p}
  =\int_{0}^{n^{p}}
    \mathsf  P\bigl(\tilde V_n^p>u\bigr)du
  =\int_{0}^{n} p s^{p-1} \mathsf P(V_n>s)ds
   + n^{p}  \mathsf P(V_n>n)
   \\\\
&  \le
  \int_{0}^{n} p s^{p-1}\bar G(s)ds
  + n^{p}\bar G(n).
\end{align*}
Therefore, 
\[
  \sum_{n=3}^{\infty}\frac{\mathsf E\tilde V_n^p}{n^{p}}
  \le
  A+B,
\]
where it was denoted
\[
  A:=\sum_{n=3}^{\infty}\frac1{n^{p}}
      \int_{0}^{n}p s^{p-1}\bar G(s)ds,
  \qquad
  B:=\sum_{n=3}^{\infty}\bar G(n).
\]

Let us evaluate for $p>1$ the remainder of the series $\sum_{n=m}^{\infty}n^{-p}$:
\[
\sum_{n=m}^{\infty} \frac{1}{n^p} \le \int_{m-1}^{\infty} \frac{1}{x^p}\,dx = \frac{(m-1)^{1-p}}{p - 1}, \qquad p > 1,\ m \ge 2.
\]
Hence, 
\begin{align*}
&  A =\sum_{n=3}^{\infty}\frac1{n^{p}}\sum_{i=1}^{n} \int^{i}_{i-1}s^{p-1} p \bar G(s) ds = \sum_{i=1}^{\infty} \left( \int_{i-1}^{i} s^{p-1} p \, \bar{G}(s) \, ds \right) \sum_{n=\max(i,3)}^{\infty} \frac{1}{n^p}
 \\\\
&  \le \sum_{i=1}^{\infty} \left( \int_{i-1}^{i} i^{p-1} p \, \bar{G}(s) \, ds \right) \sum_{n=\max(i,3)}^{\infty} \frac{1}{n^p} = \Bigg[\sum_{i=3}^{\infty} \left( \int_{i-1}^{i} i^{p-1} p \, \bar{G}(s) \, ds \right) \sum_{n=i}^{\infty} \frac{1}{n^p} 
 \\\\
& +  \left( \int_{1}^{2} 2^{p-1} p \, \bar{G}(s) \, ds \right) \sum_{n=3}^{\infty} \frac{1}{n^p} + \left( \int_{0}^{1} p \, \bar{G}(s) \, ds \right) \sum_{n=3}^{\infty} \frac{1}{n^p}\Bigg]
 \\\\
&    \le \Bigg[\sum_{i=3}^{\infty} \left( \int_{i-1}^{i} i^{p-1} p \, \bar{G}(s) \, ds \right) \frac{(i-1)^{1-p}}{p - 1} + \left( \int_{1}^{2} 2^{p-1} p \, \bar{G}(s) \, ds \right) \frac{(2)^{1-p}}{p - 1}
 \\\\
& + \left( \int_{0}^{1} p \, \bar{G}(s) \, ds \right) \frac{(2)^{1-p}}{p - 1}\Bigg]
 %\\\\& 
\le \frac{p}{p-1}\int^{\infty}_{0} \bar G(s)ds \le \frac{p}{p-1} C_G.
\end{align*}
Since the function $\bar G$ is non increasing, we get 
\(\displaystyle  \bar G(n)\le\int_{n-1}^{n}\bar G(t)dt\). 
So, 
\[
B\le\int_{2}^{\infty}\bar G(t)dt< C_G.
\]
Thus, we obtain, 
\[
  \mathsf E\sum_{n=3}^{\infty}\frac{\tilde V_n^{p}}{n^{p}}\le C_G\Bigg(\frac{p}{p-1}+1\Bigg) = C_G \Bigg(\frac{2p-1}{p-1}\Bigg),
  \qquad p>1.
\]

Let $p_k:=1/a_k>1$. Then
\[
 \mathsf  E\sum_{n=3}^{\infty}\frac{\tilde V_n^{p_k}}{n^{p_k}} \le C_G \Bigg(\frac{2p_k-1}{p_k-1}\Bigg),
  \quad\forall k\in\mathbb N.
\]

Let us choose a sequece of numbers $\{N_k\}_{k \in \mathbb N}$ so as $N_{k+1} \ge N_k+1$ and
\[
 \mathsf E\sum_{n=N_k+1}^{\infty}\frac{\tilde V_n^{p_k}}{n^{p_k}}
  <\frac1{k^{2}}.
\]
Such a sequence does exist because each series converges and its remainder goes to zero. Let 
\[
  q(n):=p_k,\quad \text{if} \quad N_k < n\le N_{k+1}.
\]
Then we have, 
\[
 \mathsf E\sum_{n=3}^{\infty}\frac{\tilde V_n^{q(n)}}{n^{q(n)}}
  \le
  \underbrace{ \mathsf E\sum_{n=3}^{N_1}\frac{\tilde V_n^{p_1}}{n^{p_1}}}_{C(p_1)}
  +\sum_{k=1}^{\infty}\frac1{k^{2}}
  <C+\frac{\pi^{2}}{6}<\infty.
\]
So, 
\[
   \mathsf  E\sum_{n=3}^{\infty}\frac{\tilde V_n^{q(n)}}{n^{q(n)}} < \infty
\]
Since $\frac{\tilde V_n}{n} \le 1$ and $q(n)\le p_n$, it follows that  $\displaystyle \Big[\frac{\tilde V_n}{n}\Big]^{q(n)} \ge \Big[\frac{\tilde V_n}{n}\Big]^{p_n}$. Therefore,
\[
 \mathsf E\sum_{n=3}^{\infty}\frac{\tilde V_n^{p_n}}{n^{p_n}} \le  \mathsf  E\sum_{n=3}^{\infty}\frac{\tilde V_n^{q(n)}}{n^{q(n)}} < \infty
\]
This implies convergence a.s. 
\[
  \sum_{n=3}^{\infty}\frac{\tilde V_n^{p_n}}{n^{p_n}}<\infty
  \quad(\text{a.s.}).
\]
Recall that $\tilde V_n=\min(V_n,n)$. Therefore, we get, 
\[
  \sum_{n=1}^{\infty}\frac{V_n^{p_n}}{n^{p_n}}
  =\sum_{n=1}^{\infty}\frac{|Y_n|^{a_n p_n}}{n^{p_n}}<\infty
  \quad(\text{a.s.}).
\]
Now, let $p_n=1/a_n$. Then we obtain 
\[
  \sum_{n=1}^{\infty}\frac{|Y_n|^{a_n/a_n}}{n^{1/a_n}} = \sum_{n=1}^{\infty}\frac{|Y_n|}{n^{1/a_n}}<\infty
  \quad(\text{a.s.}).
\]
By virtue of the Kronecker lemma (see, for example, \cite[Lemma IV.3.2]{Shiryaev2}), 
\[
\frac{1}{n^{\frac1{a_n}}} \sum_{k=1}^{n} |Y_k| \xrightarrow{\text{a.s.}} 0, \quad n\xrightarrow{}\infty
\]
Let $n = \varphi(m)$, then by virtue of the conditions (\ref{finan}) and (\ref{anlnn}) we get, 
\[
\frac{1}{m}
\big|\sum_{k=1}^{\varphi(m)} Y_k\big| \le  
\frac{C}{\varphi(m)^{1/a_{\varphi(m)}}}
\sum_{k=1}^{\varphi(m)} |Y_k| \xrightarrow{\text{a.s.}} 0, \quad m\xrightarrow{}\infty.
\]
The lemma is proved. 
\end{proof}
\begin{proof}[Proof of theorem]
The goal is to show convergence
\[
  \mathsf{P}\Bigl(\frac{1}{n}\sum_{k=1}^n Z_k \to 0,\; n\to\infty\Bigr) = 1.
\]
It follows from (\ref{le1-eq1}) and (\ref{le2-eq1}). The theorem is proved.
\end{proof}
\fi

\ifen
\section*{Acknowledgements}
For the second author this work was supported by the Foundation for the Advancement of Theoretical Physics and Mathematics ``BASIS''.

\fi

\ifru
\section*{Благодарности}
For the second author this work was supported by the Foundation for the advancement of theoretical physics and mathematics ``BASIS''.

\fi

\ifen

\fi

\ifru

\fi

\end{document}